\documentclass{amsart}
\usepackage{amssymb, latexsym, amsmath, amscd, epsfig, color}

\theoremstyle{plain}
\newtheorem{theo}{Theorem}[section]

\newtheorem{cor}[theo]{Corollary}
\newtheorem{prop}[theo]{Proposition}

\theoremstyle{definition}

\def\e{\varepsilon}

\def\vf{\varphi}

\def\d{\delta}
\def\D{\Delta}

\def\k{\kappa}
\def\DD{\mathbb D}
\def\l{\lambda}

\def\W{\mathcal W}
\def\s{\sigma}

\def\x{\times}
\def \R{\mathbb R}

\def \H{\mathbb H}

\def\om{\omega}

\def\wt{\widetilde}
\def\({\biggl(}
\def\){\biggr)}

\def\<{\bold\langle}
\def\>{\bold\rangle}

\def\M{\widetilde {M}}
\def\MM{\widehat {M}}
\begin{document}

\title{Linear drift and entropy for regular covers}
\author{Fran\c cois Ledrappier}
\address{LPMA, Bo\^ite Courrier 188, 4, Place Jussieu, 75252 PARIS cedex 05, France}\email{francois.ledrappier@upmc.fr}
\begin{abstract} We consider a regular Riemannian cover $\M$ of a compact Riemannian manifold. The linear drift $\ell$ and the Kaimanovich entropy $h$ are geometric invariants defined by asymptotic properties of the Brownian motion on $\M$. We show that  $\ell^2 \leq h$.
\end{abstract}
\maketitle

\vskip 1cm
Let $\pi : \M \to M $ be a regular Riemannian cover of a compact manifold: $\M $ is a Riemannian manifold and there is a discrete group $G$ of isometries of $\M$ acting freely and such that the quotient $M = G \setminus \M $ is a compact manifold. The quotient metric makes $M$ a compact Riemannian manifold.

We consider the Laplacian $\D$ on $\M$, the corresponding heat kernel $\wt p(t, \wt x , \wt y)$ and the associated Brownian motion $\wt X_t, t \geq 0 $.The following quantities were introduced by Guivarc'h \cite{Gu} and Kaimanovich \cite{K1}, respectively, as almost everywhere limits on the space of trajectories of the Brownian motion $\wt X$:
\begin{itemize}
\item the linear drift $ \ell \;: = \; \lim _{t \to \infty } \frac{1}{t} d_{\M} (\wt X_0, \wt X_t).$
\item the entropy $ h \; := \; \lim_{t \to \infty} -\frac{1}{t} \ln \wt p(t,\wt X_0,\wt X_t).$
\end{itemize}

In this note we prove the following 

\

{\bf Theorem A.} {\it Let $\pi : \M \to M $ be a regular Riemannian cover of a compact manifold. With the above notations, we have:}
\begin{equation}\label{main}
\ell ^2 \; \leq \; h.
\end{equation}

In the case when $\M$ is the universal covering of a compact manifold with negative curvature, inequality (\ref{main}) is due to V. Kaimanovich (\cite {K1}). Moreover in that case, there is equality in  (\ref{main}) if, and only if, the manifold $\M$ is a symmetric space of negative curvature (se the discussion below in section 1). For a general cover, it follows from (1) that, whenever $h=0$ (which is equivalent to the Liouville property of $\M$ \cite {D}, \cite {KV}), then $\ell = 0$. This was shown in \cite {KL1} by using discretization of the Brownian motion and a qualitative result for random walks. Indeed, similar quantities can be defined for a symmetric random walk on a finitely generated group, where the distance on the group is the word distance. A precise result similar to  (1) is not known for discrete random walks. There are estimates for symmetric random walks with finite support (\cite{Va}) or finite second moment (\cite{EK}).


\

Let $v$ be the volume entropy of $\M$
$$v \; = \; \lim _{R \to \infty  } \frac{\ln {\textrm {vol}} (B_{\M} (x_0, R))}{\ln R} ,$$
where $B_{\M} (x_0, R)$ is the ball of radius $R$ in $\M $ about a given point $x_0$ and vol is the Riemannian volume. It holds: $h \leq \ell v $ ( \cite {Gu}).
\begin{cor} Let $\pi : \M \to M $ be a regular Riemannian cover of a compact manifold. With the above notations, $\ell \leq v$ and $h \leq v^2$. Either equality $\ell = v, h = v^2$  implies equality in (1).
\end{cor}

Let $\l$ be the bottom of the spectrum of the Laplacian on $\M$:
$$ \l \; := \; \inf _{f \in C^2_K (\M)} \frac{\int _{\M} \| \nabla f \|^2}{\int _{\M} f^2}.$$
Clearly (by considering $C^2_K$ approximations to the functions $e^{-sd(x_o,.)}$ for $s > v/2$), we have $4 \l \leq v^2$. It can be shown that $4 \l \leq h $ (\cite {L1}, Proposition 3). Therefore,
\begin{cor} Let $\pi : \M \to M $ be a regular Riemannian cover of a compact manifold. With the above notations, equality $4 \l  = v^2$  implies equality in (1).
\end{cor}

\

Our proof of (1) is based on the construction of a compact bundle space $X_M$ over $M$ which is laminated by spaces modeled on $\M$ and of a laminated Laplacian. In the case when $M$ has negative curvature and $\M$ is the universal cover of $M$, the bundle space contains the  unit tangent bundle $T^1M$ and the lamination on $T^1M$ is the weak stable foliation of the geodesic flow. The foliated Laplacian and the associated harmonic measure are useful tools for the geometry and the dynamics of the geodesic flow (see \cite {Ga}, \cite{K1}, \cite{L2}, \cite {Y}, \cite{H}). In Section 1, we construct the lamination in the general case and state the properties of the harmonic measures which lead to Theorem A. The laminated Laplacian defines a laminated Brownian motion, a diffusion on $X_M$  with the property that the trajectories remain in the same leaf for all time. Section 2 describes this diffusion. The rest of the paper is devoted to proving propositions \ref{Furst.} and \ref{basic}.

\section{The Busemann Lamination}

We consider the {\it Busemann compactification } of the metric space $\M$:  since the space $\M$ is a complete manifold, it is a   {\it proper} metric space (closed bounded subsets are compact). Fix a point $x_0 \in \M$ and define, for $x\in \M$ the function $\xi_x (z) $ on $\M$ by:
$$ \xi_x(z) \; = \; d(x,z) - d(x, x_0).$$
The assignment $x \mapsto \xi_x $ is continuous, one-to-one and takes values in a relatively compact set of functions for the topology of uniform convergence on compact subsets of $\M$. The  Busemann  compactification  $\MM$ of $\M$ is the closure of $\M$ for that topology. The space $\MM$ is a compact separable space. The {\it Busemann boundary } $\partial \M := \MM \setminus \M$ is made of Lipschitz continuous functions $\xi$  on $\M$ such that $\xi(x_0) = 0 $. Elements of $\partial \M$ are called {\it horofunctions}. Observe that we may extend by continuity the  action of $G$ from $\M$ to $\MM$, in such a way that for $\xi $ in  $\MM$ and $g $ in $G$,
$$ g.\xi (z) \; = \; \xi (g^{-1}z) - \xi (g^{-1}(x_0)).$$

We define now the {\it horospheric suspension }  $X_M$ of $M$ as the quotient of the space $\M \times \MM$ by the diagonal action of $G$. The projection onto the first component in $\M \times \MM$ factors into a projection from $X_M$ to $M$ so that the fibers are isometric to $\MM$. It is clear that the  space $X_M$ is metric compact. If $M_0 \subset \M$ is a fundamental domain for $M$, one can represent $X_M$ as $M_0 \x \MM$ in a natural way.

To each point $\xi \in \MM$ is associated the projection $W_\xi $ of $\M \times \{\xi \}$. As a subgroup of $G$, the stabilizer $G_\xi$ of the point $\xi$ acts discretely on $\M$ and 
the space $W_\xi$ is homeomorphic  to the quotient of $\M$ by  $G_\xi$. We put on each $W_\xi$ the smooth structure and the metric inherited from $\M$.  The manifold $W_\xi $ and its metric vary continuously on $X_M$. The collection of all $W_\xi, \xi \in \MM$ form a continuous lamination $\W_M$ with leaves which are manifolds locally modeled on $\M$. In particular, it makes sense to differentiate along the leaves of the lamination and we denote $\D^{\W}$ the laminated Laplace operator acting on functions which are smooth along the leaves of the lamination. A Borel measure on $X_M$ is called {\it harmonic} if it satisfies, for all $f$ for which it makes sense, $$ \int \D^\W f dm \; = \; 0. $$
By \cite{Ga}, there exist harmonic measures and the set of harmonic probability measures is a weak* compact set of measures on $X_M$. Moreover, if $m$ is a harmonic measure and $\wt m$ is the $G$-invariant measure which extends $m$ on $\M \x \MM$, then (\cite {Ga}), there is a finite measure $\nu $ on $\MM$ and, for $\nu $-almost every $\xi$, a positive harmonic function $k_\xi (x) $ with $k_\xi (x_0) = 1$ such that the measure $m$ can be written as;
$$ \wt m \; = \; k_\xi (x) (dx \x \nu (d\xi)).$$

The harmonic probability measure $m$ is called ergodic if it is extremal among harmonic probability measures. In that case, for $\nu$-almost every $\xi$, the following limits exist along almost every trajectory of the Brownian motion (see \cite{K2} and section 3 below):
\begin{itemize}
\item the linear drift of $m$ $ \ell (m)  \; := \; \lim _{t \to \infty } \frac{1}{t} \xi(\wt X_t). $
\item the transverse entropy $k(m) \; := \; \lim _{t\to \infty} -\frac{1}{t} \ln k_\xi (\wt X_t).$
\end{itemize}

The proof of Theorem A reduces to the three following results;
\begin{prop} \label{Furst.} With the above notations, there exists an ergodic harmonic measure such that $\ell (m) = \ell $.
\end{prop}
\begin{prop} \label{basic} For all harmonic measure $m$, we have $\ell ^2(m) \leq k(m)$ with equality  only if the harmonic functions $k_\xi $ are such that $\nabla ^\W \ln k_\xi = - \ell (m)  \nabla ^\W \xi$  $m$-almost everywhere. 
\end{prop}
\begin{prop} \label{entropy} For all ergodic harmonic measure $m$, we have $k(m)  \leq h.$ \end{prop}

The proof of Proposition \ref{Furst.} is an extension of the proof of the Furstenberg formula in \cite {KL2} and is given in section 4. Kaimanovich (\cite {K1})  proved Proposition \ref{basic} under the hypothesis that the horofunctions are of class $C^2$ by applying  It\^o's formula to the function $\xi$. In the general case, horofunctions are only uniformly 1-Lipschitz, but the integrated formulas of \cite{K1} are still valid (see Section 3). See \cite{K2} and Section 3 for Proposition \ref{entropy}.

Assume that $\M$ is the universal cover of a negatively curved compact manifold $M$. Then, $\M$ is homeomorphic to an open ball and the Busemann compactification is homeomorphic to the closure of the ball. In particular, for all $x \in \M$, the Busemann boundary is homeomorphic to the unit sphere in the tangent space $T_x\M$: a unit vector $v$ defines a unique geodesic  $\s _v(t)$ such that $\s_v(0) = x, \dot {\s}_v(0) = v.$ As $t \to \infty $, $\xi _{\s_v(t)}$ converges in $\MM$ towards the Busemann function $\xi _{\s_v (+\infty)}$ and $v \mapsto \xi _{\s_v (+\infty)}$ defines the homeomorphism between $T^1_x\M$ and $\partial \M$. In particular, $\partial \M $ is a closed $G$-invariant subset of $\MM$. We can identify $\M \x \partial \M $ with the unit tangent bundle $T^1\M$. The induced action of $G$ is the natural differential action on $T^1\M$. The quotient $T^1M$ is therefore identified with a closed subset of $X_M$. For $\xi \in \partial \M$, unit vectors $v$ such that $\s_v (+\infty ) = \xi$ form a stable manifold for the geodesic flow. The lamination $\W$ in $T^1M$ is the usual stable lamination of the geodesic flow. In this case, there is a unique harmonic probability measure $m$ (see \cite{Ga}, \cite{L2}, \cite {Y}; the proof shows that any harmonic measure on $X_M$ has to be carried by $(\M \x \partial \M)/ G$) and the support of the harmonic measure is the whole  $(\M \x \partial \M)/ G$ (see \cite{A}, \cite{S}; by compactness, the curvature is pinched betwen two negative constants). Proposition \ref{Furst.} (i.e. $\ell (m) = \ell$) and \ref{basic} are due to Kaimanovich (\cite{K1}). By Proposition \ref{basic}, if we have equality in (1), then the Busemann functions are such that $\D \xi $ is a constant, the manifold $\M$ is asymptotically harmonic.  It follows then from the combined works of Y. Benoist, G.Besson, G. Courtois, P. Foulon, S. Gallot and F. Labourie (\cite {FL}, \cite{BFL}, \cite{BCG}) that the manifold $\M $ is a symmetric space.

\section{Laminated Brownian motion}

The operator $\D^\W$ is Markovian ($\D^\W 1 = 0$) and in this section, we construct the  corresponding diffusion on $X_M$. As we detail now, this diffusion  is derived from  the Brownian motion on $\M$.

We define subspaces of trajectories in $C(\R_+, M), C(\R_+, \M), C_{\MM} (\R_+, \M \times \MM) $ and $C(\R_+, \W_M)$ and natural identifications: $C(\R_+, M), C(\R_+, \M)$ are the spaces of continuous functions from $\R_+$ into  respectively $M$ and $\M$ with the natural projection from $C(\R_+,\M) $ to $C(\R_+, M)$;  the space $C_{\MM} (\R_+, \M \times \MM)$ is the space of continuous functions from $\R_+$ into $\M \times \MM$ which are {\it constant}  on the second component, with the forgetful projection from $C_{\MM} (\R_+, \M \times \MM)$ to $C(\R_+, \M)$; the group $G$ acts on $C_{\MM} (\R_+, \M \times \MM)$ by postcomposition; the quotient space of $G$-orbits in $C_{\MM} (\R_+, \M \times \MM)$ is the space $C(\R_+, \W_M)$, with the natural projection from $C_{\MM} (\R_+, \M \times \MM)$ to $C(\R_+,\W_M)$. Elements of $C(\R_+, \W_M)$ can be seen as trajectories on $X_M$ which are included in a single leaf of the lamination $\W$. Translations over $\R_+$ act by precomposition on all our spaces of trajectories and the translation by $t$ will be denoted $\s_t$ on each of them.

\

 The operator $\D$ is uniformly elliptic on $\M$. The fundamental solution of the equation $\displaystyle \frac{\partial u}{ \partial t} = \D u$ is the heat kernel $\wt p(t,\wt x, \wt y)$. There is a unique family of probabilities $\wt {P}_x, x \in \M,$ on $C(\R_+, \M)$ such that $\{\om (t), t\in \R_+ \}$ is  a Markov process with generator $\D$. This means for example that we have for $f_j \in C_c(\M), j = 0,1,2, 0<s_1< s_2,$
 \begin{eqnarray*}
 & {}& \int f_0(\om(0)) f_1(\om (s_1)) f_2(\om (s_2)) d\wt{P}_x \\ &=& \; \int f_0 (x) f_1(y_1) f_2(y_2) \wt p (s_1, x, y_1) \wt p (s_2 - s_1, y_1, y_2) dy_1 dy_2.
 \end{eqnarray*}
 
 The family of measures $\wt {P}_x, x\in \M$ defines the Brownian motion on $\M$. See e.g. \cite {P} Chapter 4.8 for the following:
\begin{prop}\label{invariance}
Let  $\wt m$ a locally finite positive measure on $\M$. The following properties are equivalent:

the measure $\wt m$ satisfies, for all $f \in C^2_c (\M)$,  $\int \D f d\wt m = 0 $,

the measure $\wt m$ is of the form $k(y)dy$ where $k$ is a positive harmonic function,

the measure $\wt m $ is $\wt {p}$ invariant, i.e. for all $t >0 $, all $f \in C^2_c (\M)$, $$ \int _{\M} \left( \int _{\M} f(y) \wt{p} (t,x, y  ) dy \right) d\wt {m} (x) \; = \; \int _{\M} f(x) d\wt  m (x), $$

the  measure  $\wt {P}_{\wt m } : =  \int _{\M} \wt{P}_x d\wt m(x)$  on $C(\R_+, \M)$ is $\s $-invariant.

\end{prop}

By uniqueness, the family of measures $\wt {P}_x$ is $G$-equivariant and projects as a family of measures $P_x, x\in M$ on $C(\R_+,M)$ which defines the Brownian motion on $M$, with the same properties as above. In particular, the heat kernel $p(u,x,y) $ on $M$ is given by 
$$ p(u,x,y) \; = \; \sum _{g \in G} \wt {p}(u, \wt {x}, g\wt {y} ), $$
where $\wt {x} ,  \wt {y} $ are lifts in $\M $ of the points $x,y$ in $M$. The Lebesgue probability measure ${\textrm {Leb}}  := \frac{1}{{\textrm  {vol}} M} {\textrm {vol}}$ on $M$ satisfies  for all $f \in C^2 (M)$,  $\int \D f d{\textrm {Leb}}  = 0 $ and  for all $t>0$, $  \int _{M} \left( \int _{M} f(y) {p} (t,x, y  ) dy \right) d{\textrm {Leb}}  (x) \; = \; \int _{M} f(x) d{\textrm {Leb}}  (x).$  Moreover, the probability measure $P= \int _M P_x d{\textrm {Leb}}  (x) $ is invariant under the time shift $\sigma$. The probability ${\textrm {Leb}} $ is the only one with any of those properties. Indeed, by Proposition \ref{invariance}, the  $G$-invariant lift of such a measure $m$ to $\M$ has to be of the form  $k(y)d{\textrm {vol}}(y)$ where $k$ is a $G$-invariant positive harmonic function, and  $G$-invariant positive harmonic functions on $\M$ are lifts of positive harmonic functions on the compact manifold $M$ and are therefore constant.

\

Let $\xi _0$ be a point in $\MM$. There is a one-to-one correspondence between  the set of trajectories in $C_{\MM}(\R_+, \M \times \MM)$ satisfying $\xi(0) = \xi_0$ (and therefore $\xi(t) = \xi_0 $ for all $t$) and $C(\R_+, \M)$. For all $x \in \M$, the measure $\wt {P}_x$ defines a measure $\wt {Q}_{x,\xi_0}$  on  the set of trajectories in $C_{\MM}(\R_+, \M \times \MM)$ satisfying $(\om (0),\xi(0) )= (x,\xi_0)$. The family $\wt {Q}_{x,\xi}$ describes the Brownian motion along the leaves of the trivial fibration of $\M \times \MM $ into $\M \times \{\xi \}$'s. In particular $\wt q(u, (x,\xi), (y,\eta)) = \wt p (u, x,y) \d_\xi (\eta) $ is the Markov kernel of the diffusion with law $\wt Q_{x,\xi}$ and we may write for $f_j \in C_c(\M \x \MM), j = 0,1,2 $ and $0<s_1< s_2,$
 \begin{eqnarray*}
 & {}& \int f_0(\om(0), \xi (0)) f_1(\om (s_1), \xi(s_1)) f_2(\om (s_2), \xi(s_2)) d\wt{Q}_{x,\xi} \\ &=& \; \int f_0 (x,\xi) f_1(y_1,\xi ) f_2(y_2, \xi ) \wt p (s_1, x, y_1) \wt p (s_2 - s_1, y_1, y_2) dy_1 dy_2.
 \end{eqnarray*}
 
 \begin{prop}\label{qinvariance}
Let  $\wt m$ a locally finite positive measure on $\M \times \MM$. The following properties are equivalent:

the measure $\wt m$ satisfies, for all $f \in C^2_c (\M \times \MM)$,  $\int \D_x f d\wt m = 0 $,

the measure $\wt m$ is of the form $k_\xi (y)dy \otimes d\nu (\xi)$ where $\nu $ is a finite measure on $\MM$, $(x,\xi) \mapsto k_\xi (x) $ is measurable and for $\nu $ almost all $\xi$, $k_\xi (y)$ is a positive harmonic function on $\M$, 

the measure $\wt m $ is $\wt {q}$ invariant, i.e. for all $t >0 $, all $f \in C^2_c (\M \times \MM)$, $$ \int _{\M \times \MM} \left( \int _{\M} f(y, \xi) \wt{p} (t,x, y  ) dy \right) d\wt {m} (x,\xi) \; = \; \int _{\M \times \MM} f(x, \xi) d\wt  m (x, \xi), $$

the  measure  $\wt {Q}_{\wt m } : =  \int _{\M \times \MM} \wt{Q}_{x, \xi} d\wt m(x, \xi )$  on $C_{\MM} (\R_+, \M \times \MM)$ is $\s $-invariant.

\end{prop}

\begin{proof}. It is clear that a measure of the form $k_\xi (y)dy \otimes d\nu (\xi)$ satisfies the other properties. Conversely, since $\wt {m}$ is locally finite, we can find a positive continuous function $b$ on $\M \times \MM$ such that $b \wt {m}$ is a finite measure. Write $b \wt{m}$ as $\int _{\MM} \wt {m}_{\xi} d\nu (\xi)$ for a finite measure $\nu $ on $\MM$ and a measurable family $ \xi \mapsto \wt {m}_{\xi }$ of probabilities on $\M$. Then, for $\nu $ almost every $\xi$, Proposition \ref {invariance} applies to the measure $b^{-1} \wt {m}_\xi $.
\end{proof}

\

In this paper, we normalize $\nu $ and the $k_\xi$s by choosing $k_\xi (x_0 ) = 1$.

Finally, the family of measures $\wt{Q}_{x,\xi}$ is $G$ equivariant, and defines a family $Q_w, w \in X_M$ of measures on $C(\R_+, \W_M)$. The family $Q_w$ describes the laminated Brownian motion. By construction,  all trajectories of the laminated Brownian motion remain on the leaf of the initial point $w(0)$. In the identification of $X_M $ with $M_0 \x \MM$, the Markov transition probabilities $q(t, (x,\xi), d(y,\eta))$ of the diffusion   with law $Q_{x,\xi}$ are
 given by:
\begin{equation*}
q(t,(x,\xi),d(y,\eta)) \; = \; \sum _{g\in G} \wt {q}(t, (\wt x,\xi), g_\star d(\wt y, \eta)) \; = \; \sum _{g\in G} \wt{p}(t,\wt x, g\wt y) d\wt y\delta _{g^{-1} \xi } (\eta),
\end{equation*}
where $\wt {x} ,  \wt {y} $ are lifts in $\M $ of the points $x,y$ in $M$.
\begin{prop}\label{harmonic} There is a one-to-one correspondence between:
 \begin{enumerate}
\item harmonic probability measures $m$ on $X_M$,
\item $G$-invariant measures $\wt m$ which satisfy the equivalent conditions of Proposition \ref{qinvariance} and such that $\wt m (M_0 \x \MM) = 1,$
 \item probability measures $m$ on $X_M$ which projects on $M$ onto the Lebesgue probability measure  and which can be written in local $\Phi ( \DD^d\times T )$ charts $m =  \int _T \left( \int k_t(x) dx\right) d\nu (t)$, where the function $(x,t) \mapsto k_t(x) $ is measurable and, for $\nu $ almost all $t \in T$, $k_t(x)$ is a positive harmonic function,
 \item probability measures $m$ on $X_M$ which can be written in a $M_0 \times \MM$ representation $$m = \int _{M_0} \left (\int _{\MM} d\mu_x (\xi) \right)  d{\textrm{Leb}} (x) ,$$ where $x \mapsto \mu_x $ is a measurable family of measures on $\MM $ with the same negligible sets and such that for almost every $\xi \in \MM$, $k_\xi (x) := \frac{d\mu_x}{d\mu_{x_0} }(\xi) $ is obtained  as the restriction to a fundamental domain of a positive harmonic function on $\M$, 
 \item probability measures $m$ on $X_M$ which are invariant under $q$: for all $f \in C^2(X_M)$, all $t >0$, we have
  $$ \int _{X_M} \left( \int _{X_M} f(y,\eta) q(t,(x,\xi), d(y,\eta) \right) dm ((x, \xi)) \; = \; \int_{X_M} f dm. $$
\item probability measures $m$ on $X_M$ such that $Q_m := \int _{X_M} Q_w dm(w) $ is a $\s$-invariant measure on $C(\R_+, \W_M)$.
\end{enumerate}
\end{prop}

\begin{proof} Let $m$ be a harmonic probability measure on $X_M$. By writing the harmonic equation for functions which are constant on the fibers, we see that the projection of $m$ onto $M$ is a harmonic probability measure and thus is ${\textrm {Leb}} $. Write $\wt m$ for the unique $G$-invariant measure on $\M \times \MM$ such that the restriction to any fundamental domain projects to $m$. The measure $\wt {m} $ satisfies  for all $f \in C^2_c (\M \times \MM)$,  $\int \D_x f d\wt m = 0 $. Conversely, the restriction of such a measure to $(M_0 \x \MM)$ is finite and harmonic. This shows the equivalence of properties (1) and (2).

Moreover, by proposition \ref{qinvariance}, the measure $\wt m$ is of the form $k_\xi (y)dy \otimes d\nu (\xi)$ where $\nu $ is a finite measure on $\MM$, $(x,\xi) \mapsto k_\xi (x) $ is measurable and for $\nu $ almost all $\xi$, $k_\xi (y)$ is a positive harmonic function on $\M$. When we restrict to the image of a $\DD^d \times T$ chart, this gives the description of property (3). Conversely, assume that $m$ satisfies property (3). Using if necessary a partition of unity we may take the function $f$  in $C^2 (X_M) $ with support inside the image of a $\Phi ( \DD^d\times T )$ chart. Then:
$$ \int \D^\W f dm = \int _T \left( \int \D_x f(x,t) k_t(x) dx\right) d\nu (t)$$
and the inner integral vanishes for all $t \in T$ such that $x \mapsto k_t(x)$ is a harmonic function, that is, for almost every $t$.

Assume $m$ satisfies property (3).  Putting together the $\DD^d \times T$ charts into a measurable $M_0 \times \MM $ representation, we have a measure which projects on a measure $\nu $ on $\MM$ and such that the conditional on $M$ are proportional to $k_\xi (x) d{\textrm {vol}}(x)$. In other words, the measure $m$ writes as $\displaystyle m = \frac {k_\xi (x)}{\int _{M_0} k_\xi (x) dx} dx \otimes d\nu.$ Since we assume that the projection onto $M_0$ is ${\textrm {Leb}} $, we can write $m = \int _{M_0} \left (\int _{\MM} d\mu_x (\xi) \right)  d{\textrm{Leb}} (x) ,$ where $$ \mu_x (d\xi) = \frac {k_\xi (x) {\textrm {vol}}M} {\int _{M_0} k_\xi (x)dx}  d\nu (\xi).$$ We indeed have $ \frac{d\mu_x}{d\mu_y }(\xi) = \frac {k_\xi (x)}{k_\xi (y)}. $ This shows that property (3) implies property (4). The converse is proven analogously,  by  setting $\nu = \mu_{x_0}$. 

Properties (5) and (6) are equivalent to (1) by general theory of diffusions with a finite invariant measure. The point to check is that a $\s$-invariant measurable set  $B$ in $C(\R_+,\W _M)$ is of the form $x \in B_0$, where $B_0$ is a $Q$-invariant subset of $X_M$. It follows from  the Markov property that  the set $B$ has $Q_w$ measure 0 or 1 for $m$-almost every $w \in X_M$. Take $B_0 = \{ w : Q_w (B) = 1\}$. 
\end{proof}

Proposition \ref {harmonic} is due to Garnett (\cite {Ga}). We included a proof in the suspension case for notational purposes. A harmonic measure is called ergodic harmonic if it cannot be decomposed into a convex combination of other harmonic measures. By proposition \ref{harmonic}, an ergodic harmonic measure is also extremal for properties (5) and (6) and therefore the time shift $\s _t $ is ergodic on $(C(\R_+,\W_M), Q_m)$. By proposition \ref{harmonic} a harmonic measure can be written, in a  $M_0 \x \MM $ representation  as 
$$\int _M d\mu_x(\xi) d{\textrm {Leb}}(x) \quad {\textrm {where }} \frac{d\mu_x}{d\mu_y }(\xi) \; = \; \frac{k_\xi (x)}{k_\xi (y)}$$ 
and $k_\xi (x) $ is a positive harmonic function for $\nu ( = \mu _{x_0} )$-almost every $\xi$. In particular, for $f \in C^2 (X_M)$ with support in the interior of $M_0$, we may write:
\begin{eqnarray*}
\int f dm \; &=& \; \int _{M_0} \int _{\MM} f(x,\xi) d\mu_x (\xi) d{\textrm {Leb}}(x) \\ &=& \int _{M_0}\left( \int _{\MM} f(x,\xi) k_\xi (x) d{\textrm {Leb}}(x) \right)  d\nu (\xi)
\end{eqnarray*}
Integrating by parts the inner integral, the following formulas follow, for all $f,g \in C^2(X_M)$:
\begin{eqnarray*} \int \D^\W f dm \; &=& \; - \int \< \nabla ^\W f, \nabla ^\W \ln k_\xi \> dm = 0 \\
\int g \D^\W f dm \; &=& \; - \int \< \nabla ^\W f, \nabla ^\W g \> dm - \int g \< \nabla ^\W f, \nabla ^\W \ln k_\xi \> dm, 
\end{eqnarray*}
where $\nabla ^\W  g$ denotes the gradient of the function $g$ along the leaves of the lamination $\W$ and $\<,\>$ the leafwise scalar product. The second formula extends by approximation to vector fields $Y$ which are $C^1 $ along the leaves and such that $Y$ and ${\textrm {div}} ^\W Y $ are continuous:
\begin{equation}\label{Green}
\int {\textrm {div}} ^\W Y dm \; = \; - \int \< Y, \nabla ^\W \ln k_\xi \> dm.
\end{equation}

\

\section{Asymptotics of harmonic measures}

In this section, we state two formulas as Proposition \ref{Vadim1} and \ref{Vadim2}. We deduce from them Proposition \ref{basic} and, using Propositions \ref{Furst.} and \ref{entropy}, Theorem A. 

Let $m$ be an ergodic harmonic measure on $X_M$. Recall that $m$ can be written as $\int _{M_0} k_\xi (x) d\nu(\xi ) d{\textrm {Leb}}(x)$ for some positive harmonic function $k_\xi  (x)$ defined for $\nu $-almost every $\xi$. The probability measure $Q_m$ is invariant and ergodic under the shift on the space of trajectories $C(\R_+, \W_M)$. There are two natural additive functional on $C(\R_+, \W_M)$ which are defined as $G$-invariant functionals on $C_{\MM}(\R_+, \M \x \MM)$: the horospherical displacement
$$ L(t,\om, \xi) : = \xi (\om(t) ) - \xi (\om (0))$$
and the harmonic kernel
$$ K(t, \om, \xi) := \ln \frac {k_\xi (\om (0))}{\k_\xi (\om (t))}.$$
The functional $K(t, \om, \xi  )$ is defined for $Q_m$-almost every $(\om, \xi )$, but  for all $t \geq 0$. We have $L(t+s, \om ,\xi) = L(t,\om , \xi) + L(s, \s_t (\om, \xi)) $ and, for $Q_m$-almost every $(\om, \xi)$, $K(t+s, \om ,\xi) = K(t,\om , \xi) + K(s, \s_t (\om, \xi)) .$

By the ergodic theorem, the two following limits exist $Q_m$-almost everywhere and are constant $Q_m$-almost everywhere:
$$ \ell (m) \; : = \; \lim_{T\to \infty} \frac{1}{T} L(T,\om, \xi) \;\; {\textrm {and }} \; \;  k (m) \; : = \; \lim_{T\to \infty} \frac{1}{T} K(T,\om, \xi) .$$
By our description of  the measure $Q_m$ in Section 2, for $\nu$-almost every $\xi$, the numbers $\ell (m)$ and $k(m)$ can also be seen as the limits along almost every trajectory  of the  Brownian motion of respectively $\frac{1}{t} \xi (\wt X_t)$ and $- \frac{1}{t} \ln k_\xi (\wt X_t)$. This is the way  they were introduced in Section 1. In particular, since the functions $\xi $ are Lipschitz, for all ergodic harmonic measure $m$, 
\begin{equation}
\ell (m) \; \leq \; \ell.
\end{equation}
The analogous result $k(m) \leq h$ is Proposition \ref{entropy}. Kaimanovich  introduced in \cite{K2} the {\it {reverse entropy}} of an ergodic harmonic measure as the number $h'(m) $ such that, for $Q_m$-almost every trajectory in $C(\R_+, \M)$, 
$$ h'(m) \; = \; \lim_{t \to \infty} - \frac{1}{t} \ln \left( \wt p (t, \om (0), \om(t) \frac {k_\xi (\om(0))}{k_\xi (\om(t))} \right).$$
Clearly, $h'(m) = h - k(m)$. Proposition \ref{entropy} follows from the observation that the number $h'(m)$ is nonnegative, since it can be seen as the entropy of a conditional process (see \cite {K1}, section 4)).

Observe that, by $\s$-invariance and ergodicity, for all $\tau >0$, we have:
$$ \ell (m) \; =\; \frac{1}{\tau} \int \big(L(\tau, \om , \xi)\big) dQ_m \quad \quad  k (m) \; =\; \frac{1}{\tau} \int \big(K(\tau, \om , \xi)\big) dQ_m. $$
For a non-ergodic harmonic measure, we {\it {define }} $\ell (m) $ and $k(m) $ by these formulas. We have:
\begin{prop} \label{Vadim1}Let $m$ be a harmonic measure. Then:
$$ k(m) \;= \; \int_{X_M} \| \nabla ^\W \ln k_\xi \|^2 dm .$$
\end{prop}
\begin{prop}\label{Vadim2} Let $m$ be a harmonic measure. Then:
$$ \ell(m) \;= \; -\int_{X_M} \< Z_\xi, \nabla ^\W \ln k_\xi \> dm ,$$
where the vector field $Z_\xi$ is defined $m$-almost everywhere by $Z_\xi := \nabla ^\W \xi.$
\end{prop}

Recall that $\xi$ is defined as the uniform limit of difference of distances. It follows that $\xi$ is 1-Lipschitz and by Rademacher Theorem, $\nabla \xi $ is defined Lebesgue-almost everywhere on $\M$. Since $m$ is harmonic, its conditional on the leaves of $\W$ are absolutely continuous, and $Z_\xi := \nabla ^\W \xi $ is defined $m$-almost everywhere. Moreover, $\|Z_\xi \| \leq 1 $ $m$-almost everywhere. Schwarz inequality,  Propositions \ref{Vadim1} and  \ref{Vadim2} yield that, for any harmonic measure $m$,
\begin{eqnarray*} \ell ^2(m) \; = \; \left| \int  _{X_M} \< Z_\xi , \nabla ^\W \ln k_\xi \> dm \right|^2 \; &\leq & \; \int _{X_M} | \< Z_\xi , \nabla ^\W \ln k_\xi \> |^2 dm \\ &\leq & \; \int_{X_M} \| \nabla ^\W \ln k_\xi \|^2 dm \; = \; k(m), 
\end{eqnarray*}
with equality only if $\nabla ^\W \ln k_\xi  =  -\ell (m) Z_\xi $ $m$-almost everywhere. 
This proves Proposition \ref{basic}. We also have:

\begin{cor}\label{final}
Let $\M  $ be a regular Riemannian cover of a compact manifold; then,
\begin{equation*}
\ell ^2 \; \leq \; h.
\end{equation*}
If there is equality $\ell^2 = h $, then there is an ergodic harmonic measure $m$ on $X_M$ such that $ \ln k_\xi (x) = -\ell \xi (x) $ $m$-almost everywhere in the case $\ell >0$, $k_\xi (x) = 1 $  $m$-almost everywhere in the case $\ell = 0 $. 
\end{cor}

\begin{proof} By Proposition \ref{basic}, for any harmonic measure $m$,  $\ell ^2(m) \leq k(m)$ with equality  only if the harmonic functions $k_\xi $ are such that,  $\nabla ^\W \ln k_\xi (x)= -\ell (m) \nabla ^\W \xi (x)$ for $m$-almost every $(x,\xi)$.  If  $\ell(m) =0 $,  $k_\xi $ is constant for $\nu $-almost every $\xi$. If $\ell (m) >0$, since both functions  $\ln k_\xi $ and $\xi $ vanish at $x_0$, $\ln k_\xi = -\ell (m) \xi$ for $\nu $-almost every $\xi$. All this applies to the measure $m_0$ given by proposition \ref{Furst.} so that:
$$ \ell ^2 \;  = \; \ell ^2(m_0) \; \leq \; k(m_0) \; \leq \; h, $$
with $\ell ^2 = h$ only if $\ell ^2 (m_0) = k(m_0) = h$ and therefore $ \ln k_\xi (x)= -\ell \xi (x) $ $m_0$-almost everywhere  in the case $\ell >0$, $k_\xi (x) = 1 $  $m_0$-almost everywhere in the case $\ell = 0 $. 

\end{proof}

Observe that in the case $\ell ^2 = h >0$, the harmonic measure given by Corollary \ref{final} gives full measure to  $\M \x \partial \M /G$ because $e^{-\ell d(x,z)}$ cannot be a harmonic function in $z$. In general the support of $m$ is smaller than  $\M \x \partial \M /G$: consider for instance $\M = \H^2 \x \H^2.$ We have $\ell ^2 = h = 2$. The space $\partial \M$ can be parametrized by $(\xi _1, \xi _2, \theta )$, where $\xi _j \in \partial \H^2$ for $j = 1,2$  and $\theta $ is an angle in $[0,\pi / 4]$; the horofunction $\xi _{\xi _1, \xi _2, \theta }$ is given by:
$$ \xi _{\xi _1, \xi _2, \theta } (z_1, z_2) \; =\; \cos \theta \; \xi _1(z_1) + \sin \theta \; \xi_2 (z_2). $$
The function $e^{- \sqrt {2} \xi } $ satisfies
$$ \D e^{- \sqrt {2} \xi _{\xi _1, \xi _2, \theta } }\;=\; \big(2 -\sqrt {2}(\cos \theta + \sin \theta)\big) e^{ - \sqrt {2} \xi _{\xi _1, \xi _2, \theta }}.$$
This is a harmonic function only if $\theta = \pi /4$. The support of the measure $m$ given by Corollary \ref{final} is included in $\M_{\pi /4} / G$, where $\M_{\pi /4} :=\{ (x, \xi); x \in \M , \xi = (\xi _1 , \xi _2 , \theta ) \in \partial \M {\textrm { and }} \theta = \pi /4 \}.$
The discussion is similar for any  symmetric space of non-positive curvature which is not of negative curvature. 

\

Theorem A is the first part of Corollary \ref{final}. When $h=0$, Corollary \ref{final} adds that there is an ergodic measure $m$ with $k_\xi(x) = 1$ $m$-almost everywhere: in terms of Proposition \ref{harmonic} (3) the measure $m$ is, in local charts, the product of the Lebesgue measure on the leaves and some transverse holonomy-invariant measure $\nu$. When $h >0$ and equality $\ell ^2 = h$ holds, one can conclude from Corollary \ref{final} that $(\MM , \nu) $ represents all bounded harmonic functions on $\M$ (cf. \cite {K1}). 

It remains to prove Propositions \ref{Furst.}, \ref{Vadim1} and \ref{Vadim2}. Proposition \ref{Furst.} is proven in Section 4. Proposition \ref{Vadim1} is due to Kaimanovich. We give a proof in section 5, because it follows the same computation as in the proof of Proposition \ref{Vadim2}  in Section 6.

\section{Proof of Proposition \ref{Furst.}}

Let $X_M$ be a horospheric suspension as above. We construct the measure $m$ by a limiting procedure (compare \cite{KL2}, proof of Theorem 7). To define a measure on $X_M$, we usually describe it as a $G$-invariant measure on $\M \times \MM$. It will project onto $M$ as Leb, and in particular, it will be a probability measure, as soon as it projects onto $\M$ as  $\wt {{\textrm {Leb}}} := \frac {dx}{{\textrm {vol}} (M)}.$  Set:
$$ \nu_t := \int _{\M} \left( \xi_\ast (\wt {p} (t, x, y ) dy ) \right)  \frac {dx}{{\textrm {vol}} (M)},$$
where, for a measure $\mu $ on $\M$,  $\xi _\ast (\mu)$ is the pushed-forward  of $\mu $ by the mapping $\xi : \M \to \MM.$ The measure $\nu_t$ is a $G$-invariant measure on $\M \times \MM$ which projects on $\wt {{\textrm {Leb}}}$ and we can write, for $f \in C_c (\M \times \MM)$, 
$$\int f(x, \xi) d\nu _t (x, \xi) \; = \; \int f(x, \xi_y) \wt {p}(t,x,y) \frac {dxdy}{{\textrm {vol}} (M)}. $$
We then form the measure $ \int \wt{q} (s, .,.) d\nu_t = \int \wt{q}(s,(x, \xi), d(y,\eta)) d\nu_t (x, \xi)$.The measure $ \int \wt{q} (s, .,.) d\nu_t  $ is a $G$-invariant measure on $\M \times \MM$ which projects on $\wt {{\textrm {Leb}}}$. Observe that $ \int \wt{q} (s, .,.) d\nu_t  = \nu _{t+s}$. Indeed, we may  write, for $f \in C_c (\M \times \MM)$, 
\begin{eqnarray*} \int f d\left( \int \wt{q} (s, .,.) d\nu_t  \right)\; & = &\; \int f(y, \eta) \wt {q}(s, (x,\xi _z), d(y,\eta)) \wt {p}(t,x,z) \frac {dzdx}{{\textrm {vol}} (M)} \\ &= &\; \int f(y, \xi_z) \wt {p}(s,x, y) \wt {p} (t,x,z) \frac {dydzdx}{{\textrm {vol}} (M)}.
\end{eqnarray*}
By the symmetry and the semigroup property of $\wt {p}$, $\int \wt{p} (s,x,y) \wt{p} (t,x,z) dx = \wt{p} (t+s, y,z)$ and we find, as claimed, 
 $$ \int f d\left(  \int \wt{q} (s, .,.) d\nu_t  \right) \; = \; \int f(y, \xi_z) \wt {p} (t+s , y, z) \frac {dydz}{{\textrm {vol}} (M)}\; = \; \int f d\nu _{t+s}.$$

The set of  measures on $X_M $ which project on Leb on $M$ is a convex weak* compact set of probability measures on $X_M$. Any limit point of $\frac {1}{T} \int_0^T \nu _t dt$ is a harmonic measure. Indeed, by the above observation 
$$\frac {1}{T} \int_0^T \nu_t dt \; = \; \frac {1}{T} \int_0^{T-1}\int \wt{q} (s, .,.) d\nu_1ds  + O(1/T),$$
so that, if $m_0 = \lim _k \frac {1}{T_k} \int_0^{T_k} \nu_t dt$, we have $$ \int \wt{q} (1, .,.) dm_0 = \lim _k \left(\frac {1}{T_k} \int_0^{T_k-1} \int \wt{q} (s+1, .,.) d\nu_1ds  + O(1/T_k)\right) \; = \; m_0.$$

Take $m_0$ such a limit. We choose a fundamental domain $M_0$ for $M$ and we compute $\ell (m_0)$:
\begin{eqnarray*}
 & & \tau \ell (m_0) = \\
 &= &\; \int \big( \xi (\omega (\tau))- \xi (\omega (0)) \big) d Q_m \\
 &=& \; \int _{M_0 \times \MM} \left(\int \big( \xi (y)- \xi (x) \big) \wt {p} (\tau , x,y) dy \right) dm_0(x,\xi)\\
 &=& \; \lim _k \frac {1}{T_k} \int_0^{T_k}  \int _{M_0 \times \MM} \left(\int \big( \xi (y)- \xi (x) \big) \wt {p} (\tau , x,y) dy \right) d\nu_t (x,\xi) dt \\
 &=& \; \lim _k \frac {1}{T_k} \int_0^{T_k}  \int _{M_0 \times \M} \left(\int \big( \xi_z (y)- \xi_z (x) \big) \wt {p} (\tau , x,y) dy \right) \wt{p}(t,x,z) \frac {dxdz}{{\textrm {vol}} (M)} dt \\
  &=& \; \lim _k \frac {1}{T_k} \int_0^{T_k}  \int _{M_0 \times \M\times \M} \big( d(z,y)- d(z,x) \big) \wt {p} (\tau , x,y)  \wt{p}(t,x,z) \frac {dxdydz}{{\textrm {vol}} (M)} dt \\
   &=& \; \lim _k \frac {1}{T_k} \int_0^{T_k} \Big( \int _{M_0 \times \M\times \M}  d(z,y) \wt {p} (\tau , x,y)  \wt{p}(t,x,z) \frac {dxdydz}{{\textrm {vol}} (M)}  \\
    & & \; \quad \quad \quad \quad \quad \quad - \int _{M_0 \times \M\times \M} d(z,x)  \wt {p} (\tau , x,y)  \wt{p}(t,x,z) \frac {dxdydz}{{\textrm {vol}} (M)} \Big) \; \; dt \\
 &=& \; \lim _k \frac {1}{T_k} \int_0^{T_k}  \Big ( \int _{M_0 \times \M}  d(z,y) \wt {p} ( t+\tau , z,y) \frac {dydz}{{\textrm {vol}} (M)} \\
 & & \; \quad \quad \quad \quad \quad \quad -  \int _{M_0 \times \M} d(z,x)    \wt{p}(t,x,z) \frac {dxdz}{{\textrm {vol}} (M)} \Big) \; \;  dt \\ 
 &=& \; \lim _k \frac {1}{T_k} \int_{0}^{ \tau } \left( \int _{M_0 \times \M}  d(x,z) \wt{p}(T_k + t,x,z) \frac {dxdz}{{\textrm {vol}} (M)} \right) dt \\
 & & \quad \quad \quad - \lim _k \frac {1}{T_k} \int_{0}^{\tau } \left( \int _{M_0 \times \M}  d(x,z) \wt{p}(t,x,z) \frac {dxdz}{{\textrm {vol}} (M)} \right) dt.
 \end{eqnarray*}
 
 The last term goes to 0 as $T_k  \to \infty$. Recall that $\ell $ is defined by the subadditive ergodic theorem so that  $\displaystyle \ell = \lim_{T \to \infty}  \frac{1}{T} \int_{M_0\times \M} d(x,z) \wt{p}(T,x,z) \frac {dxdz}{{\textrm {vol}} (M)}$. We have indeed $\tau \ell (m_0) = \tau \ell. $ The above measure $m_0$ is not necessarily ergodic, but since $\ell (m) \leq \ell $ for all harmonic measures and $m \mapsto \ell (m) $ is linear, there are ergodic measures $m$ in the extremal decomposition of $m_0$ which satisfy $\ell (m) = \ell $.
 
 \

 \section{Proof of Proposition \ref{Vadim1}}
 
 Let $m$ be a harmonic measure on $X_M$. We have to show that $ k(m) \;= \; \int_{X_M} \| \nabla ^\W \ln k_\xi \|^2 dm .$ We compute:
 \begin{eqnarray*}
 \tau k(m) \; &= &\; \int  \ln \frac {k_\xi (\om (0))}{\k_\xi (\om (\tau))} dQ_m  \\
 &=& \; \int_{M_0 \x \MM} \left( \int _{\M} \wt p(\tau,x,y) (\ln k_\xi (x) - \ln k_\xi (y)) dy \right) dm (x, \xi)
 \end{eqnarray*}
 
 The inner integral is 
 \begin{eqnarray*}
 &{}& \int _{\M} \wt p(\tau,x,y) (\ln k_\xi (x) - \ln k_\xi (y)) dy \; = \\
 &{}& \quad \quad \quad = \;  \int _{\M} \int_0^\tau \frac{\partial }{\partial s} \wt p(s ,x,y) (\ln k_\xi (x) - \ln k_\xi (y)) dsdy \\
  &{}& \quad \quad \quad = \;  \int _{\M} \int_0^\tau \D_y \wt p(s ,x,y) (\ln k_\xi (x) - \ln k_\xi (y)) dsdy \\
   &{}& \quad \quad \quad = \;   - \int_0^\tau \int _{\M}  \wt p(s ,x,y) \D_y \ln k_\xi (y) dsdy \\
     &{}& \quad \quad \quad = \;   \int_0^\tau \int _{\M}  \wt p(s ,x,y) \|\nabla _y \ln k_\xi (y) \|^2 dsdy. 
     \end{eqnarray*}
     
     Observe that the function $$ \vf (y,\xi ) =  \|\nabla _y \ln k_\xi (y) \|^2 $$ is $G$-invariant on $\M \x \MM$. Therefore the integral $$ \int _{M_0\x \MM} \left(  \int _{\M}  \wt p(s ,x,y) \vf(y,\xi) dy \right) dm(x,\xi) $$ is $\int \vf (\om (s), \xi (s) ) dQ_m .$ By invariance, we have, for all $s >0$, 
     $$  \int _{M_0\x \MM} \left(  \int _{\M}  \wt p(s ,x,y) \vf(y,\xi) dy \right) dm(x,\xi)  \; = \; \int   \|\nabla _y \ln k_\xi (y) \|^2 dm  $$ and the formula follows.
     
     \
     
     \section{Proof of Proposition \ref{Vadim2}}
     
 Recall that the horofunctions are Lipschitz, so that $Z_\xi = \nabla ^\W \xi $ exists almost everywhere along the leaves and satisfies $\|Z_\xi \|^2 \leq 1 $ $m$-almost everywhere. In particular the expression $\int_{X_M} \< Z_\xi, \nabla ^\W \ln k_\xi \> dm $ makes sense as soon as  $m$ has absolutely continuous conditional measures along the leaves. In this section, we prove that, if $m$ is a harmonic measure,  $\ell (m) = - \int_{X_M} \< Z_\xi, \nabla ^\W \ln k_\xi \> dm $.
 
 We follow the same computation as in Section 5, except that, for technical reasons we choose $\e >0$ and write:
 \begin{eqnarray*}
 \tau \ell (m) \; &= & \; \int \big(L(\e + \tau, \om , \xi) - L(\e, \om ,\xi) \big) dQ_m\\
 &=& \;  \int_{M_0 \x \MM} \left( \int _{\M} \wt p(\e + \tau,x,y) (\xi(y) - \xi (x) )dy \right) dm (x, \xi)\\
 &{}& \quad \quad \quad -  \int_{M_0 \x \MM} \left( \int _{\M} \wt p(\e,x,y) (\xi(y) - \xi (x) )dy \right) dm (x, \xi)\\
  &=& \; \int_{M_0 \x \MM} \left( \int _{\M} \wt p(\tau,x,z) (\vf _\e (z,\xi) - \vf_\e (x,\xi)) dz \right) dm (x, \xi),
  \end{eqnarray*}
  where
  $$ \vf _\e (x,\xi ) : = \int _{\M} \wt p(\e, x , y ) \xi (y) dy .$$
  
  Observe that, since the manifold $\M$ has bounded Ricci curvature, for any $s>0$ there is a constant $C(s)$ such that, for all $x, y \in \M$, $ \wt p(s,x,y) \leq C e^{-(d(x,y)/C)^2} $(\cite{CLY}). This shows that the function $\vf _\e$ is well defined and that we can separate in the above computation the integrals of $\xi (y) $ and $\xi (x)$.
  
  For all $\e >0$, the function $\vf _\e $ is smooth and satisfies:
  $$ \vf _\e (gx, g\xi ) \; = \; \vf _\e (x, \xi) - \xi (g^{-1} x_0).$$
  The inner integral is
  \begin{eqnarray*}
 &{}& \int _{\M} \wt p(\tau,x,z) (\vf _\e (z,\xi) - \vf_\e (x,\xi)) dz \; = \\
 &{}& \quad \quad \quad = \;  \int _{\M} \int_0^\tau \frac{\partial }{\partial s} \wt p(s ,x,z) (\vf _\e (z,\xi) - \vf_\e (x,\xi)) dsdz \\
  &{}& \quad \quad \quad = \;  \int _{\M} \int_0^\tau \D_z \wt p(s ,x,z) (\vf _\e (z,\xi) - \vf_\e (x,\xi)) dsdz \\
   &{}& \quad \quad \quad = \;   \int_0^\tau \int _{\M}  \wt p(s ,x,z) \D_z \vf_\e (z,\xi)  dsdz.
 \end{eqnarray*}

The function $\D_z \vf_\e (z,\xi) = {\textrm {div}}_z \nabla _z \vf_\e (z, \xi) $ is $G$-invariant and as before, we have, for all $s > 0$, 
$$  \int_{M_0 \x \MM} \left( \int _{\M} \wt p(s,x,z) \D_z \vf _\e (z,\xi)  dz \right) dm (x, \xi)  =  \int {\textrm {div}}^\W \nabla ^W \vf_\e (y, \xi) dm(y,\xi). $$
Using equation (\ref{Green}), the latter integral is $- \int \< \nabla ^\W \vf _\e, \nabla ^\W \ln k_\xi \> dm $, so that 
$$ \ell (m) \; = \; - \int \< \nabla ^\W \vf _\e, \nabla ^\W \ln k_\xi \> dm .$$

Fix $\xi$. As $\e \to 0 $, the functions $\vf _\e (x, \xi)$ are uniformly Lipschitz and converge towards $\xi$ uniformly on compact sets. Their gradients, seen as their weak gradients, converge in $L^\infty_{loc}$ towards the gradient $Z_\xi $ of the limit (\cite {EG}, Theorem 4.2.3). This proves the formula in Proposition \ref{Vadim2}, namely:
$$\ell (m) = - \int_{X_M} \< Z_\xi, \nabla ^\W \ln k_\xi \> dm .$$

\

Proposition \ref{Vadim2} was proven by Kaimanovich (\cite{K1}) with the additional hypothesis that the horofunctions are of class $C^2$. In the above proof, we can, in that case, take directly $\e = 0 $.  Recall that the horofunctions are of class $C^2$ when $M$ has nonpositive sectional curvature and $\M$ is the universal cover of $M$ (\cite{HI}).

\

{\it Acknowledgement.} We are grateful to Xiaodong Wang for his questions, interest and remarks. This work was supported in part by NSF Grant DMS-0801127 and in part by C.N.R.S.,  UMR 7599.  

\

\end{document}